\documentclass{amsart}

\usepackage[normalem]{ulem}

\usepackage{amssymb, amsmath, amsthm, hyperref, euscript}

\usepackage{amscd}

\usepackage[english]{babel}

\usepackage{bbm}

\author{Rafael Torres}

\title[Collapsing, minimal volume, entropy, and Yamabe invariant]{Some examples of vanishing Yamabe invariant and minimal volume, and collapsing of inequivalent smoothings and PL-structures}

\address{Scuola Internazionale Superiori di Studi Avanzati (SISSA)\\Via Bonomea 265\\34136\\Trieste\\Italy}

\email{rtorres@sissa.it}

\theoremstyle{plain}

\newtheorem{theorem}{Theorem}

\newtheorem{corollary}{Corollary}

\newtheorem{proposition}{Proposition}

\newtheorem{example}{Example}

\newtheorem{remark}{Remark}

\newtheorem{lemma}{Lemma}

\newtheorem{claim}{Claim}

\newtheorem{condition}{Condition}

\newtheorem{thm}{Theorem}

\newtheorem{cor}[thm]{Corollary}

\newtheorem{prop}[thm]{Proposition}

\theoremstyle{definition}

\newtheorem{definition}{Definition}

\newcommand{\R}{\mathbb{R}}

\newcommand{\Z}{\mathbb{Z}}

\newcommand{\N}{\mathbb{N}}

\newcommand{\C}{\mathbb{C}}

\newcommand{\Aut}{\textnormal{Aut}}

\newcommand{\id}{\textnormal{id}}

\DeclareMathOperator{\MinVol}{MinVol}
\DeclareMathOperator{\DMinVol}{D-MinVol}
\DeclareMathOperator{\Vol}{Vol}
\DeclareMathOperator{\He}{h}
\DeclareMathOperator{\Id}{id}
\DeclareMathOperator{\Diff}{Diff}

\begin{document}

\maketitle

%\tableofcontents

\emph{Abstract}: In this short note, exploits of constructions of $\mathcal{F}$-structures coupled with technology developed by Cheeger-Gromov and Paternain-Petean are seen to yield a procedure to compute minimal entropy, minimal volume, Yamabe invariant and to study collapsing with bounded sectional curvature on inequivalent smooth structures and inequivalent PL-structures within a fixed homeomorphism class. We compute these fundamental Riemannian invariants for every high-dimensional smooth manifold on the homeomorphism class of any smooth manifold that admits a Riemannian metric of zero sectional curvature. This includes all exotic and all fake tori of dimension greater than four. We observe that the minimal volume is not an invariant of the smooth structures, yet the Yamabe invariant does discern the standard smooth structure from all the others. We also observe that the fundamental group places no restriction on the vanishing of the minimal volume and collapse with bounded sectional curvature for high-dimensional manifolds.

\section{Introduction and main results}

Within a fixed homeomorphism class, we refer as inequivalent smoothings to different smooth structures up to diffeomorphism and as inequivalent PL-structures to different PL-structures up to PL-homeomorphism. The first examples of inequivalent smoothings were provided by Milnor \cite{[Mil]} on the homeomorphism class of the 7-sphere. A classical success of surgery theory established the existence of both inequivalent smoothings and inequivalent PL-structures on high-dimensional manifolds \cite{[W]}. A manifold homeomorphic to an n-sphere is known as a homotopy sphere. Hsiang-Wall \cite{[HW]} have shown that a closed n-manifold that is homotopy equivalent to the n-torus is homeomorphic to it provided $n\geq 5$. Such a manifold is known as a homotopy n-torus. Farrell-Jones \cite{[FJo]} have proven the Borel conjecture for closed Riemannian manifolds with non-negative sectional curvature. Motivated by their result, we call a homotopy Euclidean space n-form a closed n-manifold that is homotopy equivalent to $\mathbb{E}^n/\Gamma$.\\

It is known that inequivalent smoothings need not share certain basic geometric properties with the standard smooth structure. Hitchin has shown in \cite{[NH]} that certain homotopy spheres do not admit a Riemannian metric of positive scalar curvature, while the round metric on the standard $S^n$ has positive sectional curvature. When it comes to computations of fundamental invariants in Riemannian geometry, Euclidean geometry canonically serves as a basic model; see Section \ref{Section CG} for definitions of the invariants studied in this paper. For example, among the many interesting geometric traits of an  n-dimensional compact Euclidean space form $\mathbb{E}^n/\Gamma$ are the existence of Riemannian metrics of zero sectional curvature, hence the vanishing and realization of its Yamabe invariant. Gromov's minimal volume is zero for $\mathbb{E}^n/\Gamma$, and so is its minimal entropy. Moreover, Riemannian manifolds with zero sectional curvature make prototypes to study collapsing since they observe that collapse by rescalling preserves bounded sectional curvature. Although these invariants play a fundamental role in Riemannian geometry, their computations are as challenging as they are scarce.\\

We observe in this note that exploits of $\mathcal{F}$-structures and results of Gromov \cite{[G]}, Cheeger-Gromov \cite{[CG]}, and Paternain-Petean {[PP1]}, among others, yield a procedure to calculate the aforementioned invariants on inequivalent smoothings and PL-structures for manifolds of dimension at least five. This procedure enlarges considerably the set of examples for which the values of the invariants are known. We show that the minimal volume does not distinguish different smooth nor PL-structures on a homotopy Euclidean space form, while the Yamabe invariant does discern the standard smooth structure. The precise statement is as follows.

\begin{thm}
{\label{Theorem P}} Every homotopy Euclidean space n-form $M$ with $n\geq 5$ admits a polarized $\mathcal{F}$-structure. Consequently, \begin{equation}\MinVol(M) = 0,\end{equation} and $M$ collapses with bounded sectional curvature. 

The Yamabe invariant satisfies\begin{equation}\mathcal{Y}(M) = 0,\end{equation} and it is realized, i.e., there exists a scalar-flat Riemannian metric on $M$ if and only if $M$ is diffeomorphic or PL-homeomorphic to the Euclidean space form $\mathbb{E}/\Gamma$.
\end{thm}

Theorem \ref{Theorem P} can be compared to results in \cite{[B]} and \cite{[BaT]} (see Remark \ref{Remark NilDavis} for a suggested extension). Bessi\`eres has shown that the minimal volume is sensitive to inequivalent smoothings on homeomorphism types of hyperbolic manifolds \cite{[B]}. Inequivalent smoothings and PL-structures on the n-torus for were shown to exist by Casson, Wall, and Hsiang-Shaneson \cite{[HS], [HW]}, provided $n\geq 5$. They are respectively known as exotic and fake tori, and they are all addressed by Theorem \ref{Theorem P}. The torus $T^n = S^1\times \cdots \times S^1$ is called the standard n-torus for $n\in \N$. 

\begin{cor}{\label{Corollary P}} Every homotopy n-torus $M$ with $n\geq 5$ admits a polarized $\mathcal{F}$-structure. Consequently, \begin{equation}\MinVol(M) = 0,\end{equation} and $M$ collapses with bounded sectional curvature. 

The Yamabe invariant satisfies \begin{equation} \mathcal{Y}(M) = 0,\end{equation} and it is realized, i.e., there exists a scalar-flat Riemannian metric on $M$ if and only if $M$ is diffeomorphic or PL-homeomorphic to the n-torus $S^1\times \cdots \times S^1$.

\end{cor}

Baues-Tuschmann have shown that the vanishing of the more restrictive invariant $\DMinVol$ does discern the standard smooth structure on the n-torus from certain inequivalent smoothings (see Section \ref{Section CG} for the definition of the invariant). Corollary \ref{Corollary P} says this is not the case for $\MinVol$ of any homotopy n-torus. Moreover, we observe the following corollary of Baues-Tuschmann's work.

\begin{prop}{\label{Proposition P}} Within the homeomorphism type of the n-torus for $n\geq 5$ there are smoothings and PL-structures for which $\DMinVol$ decreases under finite coverings.

On the other hand, for every choice of smoothing or PL-structure, $\MinVol$ is invariant under finite coverings.

\end{prop}

Corollary \ref{Corollary P} answers a question of Baues-Tuschmann \cite{[BaT]}; see Remark \ref{Remark TS}. Theorem \ref{Theorem P} and the results in this note can be generalized to many more inequivalent smoothings and PL-structures on a broader range of homeomorphism types (see Corollary \ref{Corollary HRPN}, Theorem \ref{Theorem HP}, and Remark \ref{Remark NilDavis}).\\

%Results of Montgomery-Yang \cite{[MoYa]} and Cheeger-Gromov \cite{[CG]} are coupled in Theorem \ref{Theorem S7} to observe that every smoothing of the 7-sphere has vanishing minimal volume and collapses with bounded sectional curvature. It is proven in Theorem \ref{Theorem PolarizedHSB} that the same holds for odd-dimensional homotopy spheres that bound parallelizable manifolds. Similar computations on homotopy real projective n-spaces are considered in Corollary \ref{Corollary HRPN} and Corollary \ref{Corollary PolarizedHP}, including all four homotopy $\R P^5$'s. We point out in Theorem \ref{Theorem HP} that results of Hitchin \cite{[NH]} and Petean \cite{[JP]} imply that every homotopy sphere that does not admit a Riemannian metric of positive scalar curvature has vanishing and unrealized Yamabe invariant. Proposition \ref{Proposition Proj} samples another extension on a smoothing on the complex projective 5-space.\\ 

We conclude the note by showing that the fundamental group places no restriction on the vanishing of the minimal volume and collapse with bounded sectional curvature for manifolds of dimension greater than four. 

\begin{thm}{\label{Theorem FundGroup}} Let $G$ be a finitely presented group and suppose $n\geq 5$. There exists a closed smooth $n$-manifold $M(G)$ with fundamental group $\pi_1(M(G))\cong G$ that admits a polarized $\mathcal{T}$-structure. Consequently,\begin{equation}\MinVol(M(G)) = 0,\end{equation} and $X(G)$ collapses with bounded sectional curvature. 
\end{thm}

Cheeger-Rong \cite[Corollary 0.3]{[CR]} have shown that in some sense most manifolds with bounded sectional curvature have minimal volume zero.\\

The note is structured as follows. The definitions, machinery and main results that we build upon are recollected in Section \ref{Section Definitions}. Section \ref{Section TC} contains the definition of $\mathcal{F}$-structures, the principal tool of this paper, as well as several useful examples and constructions that include all $\mathcal{F}$-structures of Theorem \ref{Theorem P}, its corollary, and Theorem \ref{Theorem FundGroup}. The definitions of the Riemannian invariants considered in this paper and the results we build upon are collected in Section \ref{Section CG}. The vanishing and non-realization of the Yamabe invariant follows from the main results of Section \ref{Section CG} and Section \ref{Section NY}. Existence of $\mathcal{F}$-structures on exotic spheres and homotopy real projective spaces is considered in Section \ref{Section Spheres}. Homotopy Euclidean space forms are studied in Section \ref{Section Tori}. It contains a description of the results of surgery theory that we employ. The proofs of Theorem \ref{Theorem P}, Corollary \ref{Corollary P}, Proposition \ref{Proposition P} and Theorem \ref{Theorem FundGroup} are given in Section \ref{Section Proof P}.

\section{Definitions and background results}{\label{Section Definitions}}

We collect in this section the definitions and main results used in this manuscript in order to make it as self-contained as possible. 

\subsection{$\mathcal{F}$-structures and constructions}{\label{Section TC}} The definition of an $\mathcal{F}$-structure below was given in \cite[Section 2]{[PP3]} and it is equivalent to the original one introduced by Cheeger-Gromov \cite{[G], [CG]}.

\begin{definition}{\label{Definition T-str}}An \emph{$\mathcal{F}$-structure} on a smooth closed manifold $M$ is given by 
\begin{enumerate}
\item a finite open cover $\{U_1, \ldots, U_N\}$ of $M$;
\item a finite Galois covering $\pi_i: \widetilde{U_i}\rightarrow U_i$ with $\Gamma_i$ a group of deck transformations for $1\leq i \leq N$;
\item a smooth effective torus action with finite kernel of a $k_i$-dimensional torus \begin{center}
$\phi_i: T^{k_i}\rightarrow \Diff(\widetilde{U_i})$\end{center}for $1\leq i \leq N$;
\item a representation $\Phi_i: \Gamma_i \rightarrow \Aut(T^{k_i})$ such that \begin{center}$\gamma(\phi_i(t)(x)) = \phi_i(\Phi_i(\gamma)(t))(\gamma x)$ \end{center}for all $\gamma \in \Gamma_i$, $t\in T^{k_i}$, and $x\in \widetilde{U}_i$;
\item for any subcollection $\{U_{i_1}, \ldots, U_{i_l}\}$ that satisfies \begin{center}$U_{i_1\cdots i_l}:= U_{i_1} \cap \cdots\cap$ $U_{i_l} \neq \emptyset$,\end{center} the following compatibility condition holds: let $\widetilde{U}_{i_1\cdots i_l}$ be the set of all points $(x_{i_1}, \ldots, x_{i_l}) \in \widetilde{U}_{i_1}\times \cdots \times \widetilde{U}_{i_l}$ such that $\pi_{i_1}(x_{i_1}) = \cdots = \pi_{i_l}(x_{i_l})$. The set $\widetilde{U}_{i_1\cdots i_l}$ covers $\pi^{-1}(U_{i_1\cdots i_l}) \subset \widetilde{U}_{i_1\cdots i_l}$ for all $1\leq j \leq l$. It is required that $\phi_{i_j}$ leaves $\pi^{-1}_{i_j}(U_{i_1\cdots i_l})$ invariant, and it lifts to an action on $\widetilde{U}_{i_1\cdots i_l}$ such that all lifted actions commute.
\item An $\mathcal{F}$-structure is called a \emph{$\mathcal{T}$-structure} if the Galois coverings $\pi_i: \widetilde{U}_i \rightarrow U_i$ in Item (1) can be taken to be trivial for every $i$.
\item The $\mathcal{F}$-structure is said to be polarized if the torus action $\phi_i$ of Item (3) is fixed-point free $\forall$ open set $U_i$.
\item The $\mathcal{F}$-structure is pure if the dimension of the orbit of the torus action $\phi_i$ of Item (3) is constant for every point $x\in \widetilde{U}_i$ and every $i$.
\end{enumerate}
\end{definition}

A non-trivial circle action yields a $\mathcal{T}$-structure. We now discuss two warm-up examples.

\begin{example}\label{Example SR}Non-pure $\mathcal{T}$-structure on $S^n$ and non-pure $\mathcal{F}$-structure on $\R P^n$ that are polarized if $n$ is odd. \emph{Write the n-sphere as \begin{equation}\label{Decomposition S1}S^n = \partial D^{n + 1} = \partial D^2\times D^{n - 1} = (S^1\times D^{n - 1}) \cup_{\Id} (D^2\times S^{n - 2})\end{equation}where the pieces on the right hand side of (\ref{Decomposition S1}) are identified along the common boundary $S^1\times S^{n - 2}$ by the identity map. We use this decomposition of $S^n$ to construct a $\mathcal{T}$-structure as follows. Define as the covers of Items (1) and (2) of Definition \ref{Definition T-str}\begin{equation}\label{Set1}\widetilde{U}_1 = U_1:= S^1\times D^{n - 1}\end{equation}and\begin{equation}\label{Set2}\widetilde{U}_2 = U_2:= D^2\times S^{n - 2}.\end{equation}These choices imply that Item (4) does not need to be checked. Regarding Item (3), we have the following. The circle acts by rotations on the $S^1$ factor of $U_1$, and we label it \begin{equation}\label{CircEx1}\phi_1:S^1\rightarrow \Diff(U_1).\end{equation}This is a free action. Take a circle action on the $S^{n - 2}$-factor of $U_2$ and label it \begin{equation}\label{CircEx2}\phi_2:S^1\rightarrow \Diff(U_2).\end{equation}If $n$ is odd, then $\phi_2$ can be chosen to be fixed point free, i.e., as the Hopf action. There is a $T^2$-action on the common boundary induced by $\phi_1$ and $\phi_2$, hence Item (5) is verified. Our choices (\ref{Set1}) and (\ref{Set2}) comply with Item (6) in the definition. We have thus constructed a non-pure $\mathcal{T}$-structure on the n-sphere that is polarized in the odd-dimensional case.}

\emph{The real projective n-space is the orbit space $S^n/T$ of the antipodal involution $T: S^n\rightarrow S^n$ that identifies $x\mapsto -x$ for $x\in S^n$. The ecomposition (\ref{Decomposition S1}) yields \begin{equation}\R P^n = (D^{n - 1}\widetilde{\times} S^1) \cup_{\Id} (D^2\widetilde{\times} \R P^{n - 2})\end{equation}We point out the following abuse of notation that will carry for the remaining discussion of this example. If $n$ is even, the symbol $D^{n - 1}\widetilde{\times} S^1$ denotes the nonorientable (n - 1)-disk bundle over the circle, and $D^2\widetilde{\times} \R P^{n - 2}$ is the twisted nonorientable 2-disk bundle over the real projective (n - 2)-space. In the orientable case, i.e., when $n$ is odd, the symbol $D^{n - 1}\widetilde{\times} S^1$ simply denotes the trivial bundle $D^{n - 1}\times S^1$, while $D^2\widetilde{\times} \R P^{n - 2}$ is the twisted orientable bundle. A construction of an $\mathcal{F}$-structure on $\R P^n$ suggested by these decompositions is as follows. Define\begin{equation}U_1:=D^{n - 1}\widetilde{\times} S^1 = (S^1\times D^{n - 1})/T\end{equation} and \begin{equation}U_2:= D^2\widetilde{\times} \R P^{n - 2} = (D^2\times S^{n - 2})/T.\end{equation}Their covers $\widetilde{U}_1$ and $\widetilde{U}_2$ are respectively given in (\ref{Set1}) and (\ref{Set2}) with $\Gamma_i = \Z/2$ as in Items (1) and (2) in Definition \ref{Definition T-str}. We choose the same circle actions $\phi_1$ and $\phi_2$ as in (\ref{CircEx1}) and (\ref{CircEx2}), and Items (3) and (5) hold. The Deck transformation given by the involution $T$ is the antipodal map on each piece. We use it to definite the homeomorphism $\Phi_i: \Z/2\rightarrow \Aut(S^1) \cong \Z/2$ as $t\mapsto - t$ for $t\in \Gamma_i$ as in Item (4). Moreover, both $\phi_1$ and $\phi_2$ commute with $T$. If $n$ is odd, we can choose $\widetilde{U}_1 = U_1$ and $\phi_2$ is free.  We have then constructed a non-pure $\mathcal{F}$-structure on the real projective n-space that is polarized if its dimension is odd.}

\end{example}

\begin{example}{\label{Example MY}}Pseudofree circle actions give rise to polarized $\mathcal{T}$-structures. \emph{A circle action $\phi:S^1\rightarrow \Diff(M)$ on a smooth manifold $M$ is said to be pseudofree if it is not free, every orbit is one-dimensional, and if the isotropy group of $\phi$ is the identity except for isolated exceptional orbits, where it is a finite cyclic group \cite{[MoYa]}. Since all isotropy groups of such an action are discrete, the action is locally-free.}
\end{example}

Our constructions of $\mathcal{F}$-structures in this section continue with a proof of the first claim of Theorem \ref{Theorem FundGroup}. 

\begin{proposition}{\label{Proposition PGPolarized}} Let $G$ be a finitely presented group and suppose $n\geq 5$. There exists a closed smooth $n$-manifold $X(G)$ with fundamental group $\pi_1(X(G))\cong G$ that admits a polarized $\mathcal{T}$-structure.
\end{proposition}

A classical result of Dehn states the existence of a smooth closed orientable $n$-manifold $X(G)$ whose fundamental group is isomorphic to any finitely presented group $G$ for every integer $n\geq 4$ (cf. \cite{[MD]}). It trivially follows that for any $k\geq 7$ and $k\neq 8$ there is a $k$-manifold $M(G)$ with arbitrary finitely presented fundamental group that admits a polarized $\mathcal{T}$-structure by taking \begin{equation}{\label{ProductTrivial}}M(G):= X(G)\times \underbrace{S^3\times \cdots \times S^3}_{r} \times \underbrace{S^2\times \cdots \times S^2}_{s}\end{equation} with $r\in \N$ and $s\in \Z_{\geq 0}$ and equip it with a free circle action on a 3-sphere factor. In order to prove Proposition \ref{Proposition PGPolarized}, we need to construct such a manifold in dimensions five, six, and eight. We proceed to do so using cut-and-paste constructions in every dimension, thus yielding a different manifold and $\mathcal{T}$-structure to the one described in (\ref{ProductTrivial}). We give two constructions of $\mathcal{T}$-structures arising from surgery along loops and from surgery along embedded tori. Each piece in the assembly is equipped with a torus action and all these actions commute along the common boundary where the pieces are glued together as it was sampled in Example \ref{Example SR}.

\begin{proof}  We first discuss the construction on an odd-dimensional manifold with arbitrary finitely presented fundamental group using surgery along loops. Let $X(G)$ be a closed smooth $(n - 1)$-manifold of even dimension and $\pi_1(X(G))\cong G$ \cite{[MD]} and take the odd-dimensional $n$-manifold $X(G)\times S^1$ with fundamental group $G\times \Z$ for $n\geq 5$. A tubular neighborhood of the circle factor $\{x\}\times S^1\subset X(G)\times S^1$ is diffeomorphic to $D^{n - 1}\times S^1$. We carve out this tubular neighborhood and construct the closed smooth $n$-manifold\begin{equation}{\label{Gluing}} M(G):= (X(G)\times S^1 \backslash D^{n - 1}\times S^1) \cup_{\Id} (S^{n - 2}\times D^2),\end{equation} where the pieces are identified along their common $S^{n - 2}\times S^1$ boundaries using the identity map. The Seifert-van Kampen theorem implies that the fundamental group of $M(G)$ is isomorphic to $G$. The isotopy class of the loop $\{x\}\times S^1\subset X(G)\times S^1$ is the representative for the homotopy class of the generator of the infinite cyclic factor in the group $\pi_1(X(G)\times S^1) = G\times \Z t$. The induced identification from (\ref{Gluing}) sets $t = 0$ in $\pi_1(M(G))$ and this group is isomorphic to $G$.  

The choice of coverings and circle actions of Definition \ref{Definition T-str} are the following. We pick the first piece of $M(G)$ as the complement of the aforementioned tubular neighborhood inside $X(G)\times S^1$, and set\begin{equation}\widetilde{U}_1 = U_1:= X(G)\times S^1\backslash(D^{n - 1}\times S^1).\end{equation} It has an induced circle action $\phi_1: S^1\rightarrow \Diff(U_1)$ from the free circle action on the $S^1$ factor of $X(G)\times S^1$. The second piece is chosen as\begin{equation}\widetilde{U}_2 = U_2:= S^{n - 2}\times D^2.\end{equation} We take a free circle action $\phi_2:S^1\rightarrow \Diff(S^{n - 2}\times D^2)$ induced from a free circle action on the sphere $S^{n - 2}\times \{d\}\subset S^{n - 2}\times D^2$ for $d\in D^2$, which is odd-dimensional. This is the only part of the proof that requires $n$ to be odd. These considerations verify Items (1), (2), and (3) of Definition \ref{Definition T-str}. Notice that since the Galois covering of Item (2) is trivial, Item (4) holds automatically. Item (5) holds due to our choices of circle actions. There is a $T^2$-action $(\phi_2, \phi_1)$ on the $S^{n - 2}\times S^1$ boundary of each piece; each circle action acts on a factor. Our choice of covering satisfies condition of Item (6). Since both actions $\phi_1$ and $\phi_2$ are fixed-point free, the condition in Item (7) holds and we have constructed a polarized $\mathcal{T}$-structure on an n-dimensional manifold $M(G)$ with finitely presented fundamental group $G$ for every odd dimension $n\geq 5$. 

We finish the proof of Proposition \ref{Proposition PGPolarized} describing the construction of $\mathcal{T}$-structures using surgeries along embedded tori. This cut-and-paste technique not only settles the even-dimensions case, but it also provides a proof of the proposition in all dimensions. We describe the construction for the 6-dimensional case for the sake of a clear exposition, and we proceed to do so. Gompf \cite[Theorem 10.2.10]{[Go]} constructed a closed smooth 4-manifold $Y(G)$ with $\pi_1(Y(G)) \cong G$ and that contains two disjoint 2-tori $\{T_1, T_2\}$ of self-intersection zero. Their tubular neighborhoods $\nu(T_i)$ $i = 1, 2$ are diffeomorphic to $D^2\times T^2$ and since each torus intersects a 2-sphere inside $Y(G)$ \cite[Theorem 10.2.10]{[Go]}, we have $\pi_1(Y(G) \backslash (\nu(T_1)\cup \nu(T_2)))\cong \pi_1(Y(G))$. Consider the 6-dimensional manifold $Y(G)\times S^1\times S^1$ with fundamental group $G\times \Z \times \Z$. 

The manifold $M(G)$ is obtained by carving out two disjoint embedded 4-tori\begin{equation}T^4_i:= T^2_i\times S^1\times S^1\end{equation} inside $Y(G)\times S^1\times S^1$ and gluing them back in with an specified choice of diffeomorphism of the common boundary. A tubular neighborhood $\nu(T^4_i)$ is diffeomorphic to\begin{equation}D^2 \times T^2\times S^1\times S^1\end{equation} since $T^4_i$ has zero self-intersection for $i = 1, 2$. We build our manifold as\begin{equation}\label{CPConstruction}M(G):= (Y(G)\times S^1\times S^1\backslash\nu(T^4_1)\cup \nu(T^4_2))\cup (\underset{i= 1}{\overset{2}\bigsqcup} (D^2_i\times T^2\times S^1\times S^1))\end{equation} with the following choices of diffeomorphisms\begin{equation}{\label{GluingTori}}\varphi_i:\partial D_i^2\times T^2\times S^1 \times S^1\longrightarrow \partial D^2\times T^2\times S^1\times S^1\end{equation}of the boundary. Notice that the subindex associated to the 2-disk as $D^2_i$ for $i = 1, 2$ discerns the carved torus $T^4_i$ that is being glued back in. We choose the diffeomorphism $\varphi_1: \partial D_1^2\times T^2\times S^1\times S^1 \rightarrow \partial D^2\times T^2\times S^1\times S^1$ that identifies $\partial D^2$ with the circle factor $\{\theta_0\}\times \{t\}\times S^1\times \{\theta_2\}$ for $\theta_j\in S^1$ and $t\in T^2$. Analogously, we use the diffeomorphism $\varphi_2: \partial D^2_2\times T^2\times S^1\times S^1 \rightarrow \partial D^2\times T^2\times S^1\times S^1$ that identifies $\partial D^2$ with the circle factor $\{\theta_0\}\times \{t\}\times \{\theta_1\}\times S^1$. These operations are high-dimensional generalization of the multiplicity zero logarithmic transformations of 4-manifolds \cite{[Go]}.

We use the Seifert-van Kampen theorem to conclude that the fundamental group of $M(G)$ is isomorphic to $G$ as follows. Both meridians of the 4-tori that we removed are nullhomotopic in the complement and thus we have an isomorphism $\pi_1(Y(G)\times S^1\times S^1\backslash (\nu(T^4_1)\cup \nu(T^4_2)) \cong \pi_1(Y(G)))\times \Z t \times \Z s$. Each isotopy class of the two loops $\{y\}\times S^1\times \{\theta_2\}\subset Y(G)\times S^1\times S^1$ and $\{y\}\times \{\theta_1\}\times S^1\subset Y(G)\times S^1\times S^1$ is the representative for the homotopy class of the generator $t$ and $s$of the infinite cyclic factor in the group $\pi_1(Y(G)\times S^1\times S^1) = G\times \Z t \times \Z s$. The induced identifications from the specified diffeomorphisms (\ref{GluingTori}) set $t = 0$ and $s = 0$ in the group $\pi_1(M(G))$, which implies that it is isomorphic to $G$.  

We construct the polarized $\mathcal{T}$-structure as follows. The open sets of the coverings of Items (1) and (2) in Definition \ref{Definition T-str} are defined following the deconstruction of $M(G)$ in (\ref{CPConstruction}) as\begin{equation}\widetilde{U}_1 = U_1:= Y(G)\times S^1 \times S^1\backslash((D^2\times T_1\times S^1\times S^1) \cup (D^2\times T_2\times S^1\times S^1)),\end{equation}\begin{equation}\widetilde{U}_2 = U_2:= D^2\times T_1\times S^1\times S^1 = D^2\times T^2\times S^1\times S^1,\end{equation}and\begin{equation}\widetilde{U}_3 = U_3:= D^2\times T_2\times S^1\times S^1  = D^2\times T^2\times S^1\times S^1.\end{equation} We choose a free circle action on each $U_i$ as follows. On $U_1$, we pick the free circle action $\phi:S^1\rightarrow \Diff(U_1)$ induced by the free action of the circle factor $\{y\}\times S^1\times \{\theta_2\}\subset Y(G)\times S^1\times S^1$ for $y\in Y(G)$ and $\theta_2\in S^1$ acting on itself. We choose $\phi_2: S^1\rightarrow \Diff(D^2\times T^2\times S^1\times S^1)$ to be a circle action on the first circle of the 2-torus factor of $\{d\}\times T^2 \times \{\theta_1\}\times \{\theta_2\} = \{d\}\times (S^1\times S^1)\times \{\theta_1\}\times \{\theta_2\}\subset U_2$ for $d\in D^2$ and $\theta_j\in S^1$. Similarly, $\phi_3: S^1\rightarrow \Diff(D^2\times T^2\times S^1\times S^1)$ is chosen to be the circle action of the second circle of the $T^2$ factor in $U_3$ acting on itself. This verifies Item (3) in Definition \ref{Definition T-str}, and Item (4) does not need to be considered in this case since $\Gamma_i = \{1\}$.

The actions $\phi_i$ for $i = 1, 2, 3$ paste together to a $T^3$-action on the boundary $S^1\times T^2\times S^1\times S^1$, hence the condition of Item (5) is met and so is the on of Item (6). Since all three actions are free, Item (7) is satisfied and we have constructed a polarized $\mathcal{T}$-structure on a closed smooth 6-manifold $M(G)$ with fundamental group $G$. As we have mentioned, the argument presented is readily extended to dimension greater or equal to five. This concludes the proof of the proposition.
\end{proof}

Soma \cite{[TSo]} showed that existence of polarized $\mathcal{T}$-structures on 3-manifolds is closed under connected sums, and Gromov \cite[Appendix 2]{[G]} (cf. \cite[Example A.1]{[CG]}) remarked that the result holds in all odd dimensions. With broader generality, Paternain-Petean have shown that the existence of $\mathcal{T}$-structures is closed under connected sums of closed manifolds of dimensions greater than two \cite[Theorem 5.9]{[PP1]}.

\begin{proposition}{\label{Proposition Connected}} Soma, Gromov. Let $X$ and $Y$ be closed smooth $(2k + 1)$-manifolds with $k\in \N$ that admit a polarized $\mathcal{F}$-structure. Assume that the open cover of each $\mathcal{F}$-structure contains at least one open set with trivial Galois covering. Their connected sum $X\#Y$ admits a polarized $\mathcal{F}$-structure.
\end{proposition}

The following proof of Proposition \ref{Proposition Connected} is due to to Paternain-Petean \cite[Proof of Theorem 5.9]{[PP1]}; it is a word by word repetition of their argument modulo a small tweak to justify the addition of the adjective 'polarized'. The proof is included for the sake of completeness.

\begin{proof} Paternain-Petean's construction of the polarized $\mathcal{F}$-structure on $X\#Y$ is as follows. The connected sum $X\# Y$ is deconstructed into a union of three pieces. Each piece is equipped with a polarized $\mathcal{F}$-structure such that the $\mathcal{T}$-structures induced on adjacent boundary components are compatible with each other (see Definition \ref{Definition T-str}). The three polarized $\mathcal{F}$-structures then glue together to a (global) polarized $\mathcal{F}$-structure on $X\# Y$. 

Begin by finding embedded solid tori \begin{center} $S^1\times D_x\hookrightarrow X$ and $S^1\times D_y\hookrightarrow Y$,\end{center} where $D_x$ and $D_y$ are small $2k$-balls centered at points $x\in X$ and $y\in Y$ respectively. We describe this step of the procedure just for $X$ given that it is the very same for $Y$. If needed, modify the $\mathcal{F}$-structure on $X$  and take a point $x\in X$ contained in a single open set $U_j$ (for a fixed $j$) of the polarized $\mathcal{F}$-structure that has trivial Galois covering as in Item (2) of Definition \ref{Definition T-str}. Without loss of generality, we can assume that the corresponding torus $T^{k_j}$ acting on $U_j$ is a circle, and that the point $x$ is contained in a regular orbit. The small $2k$-ball $D_x$ centered at $x$ is taken to be transverse to this circle action. The embedded solid torus $S^1\times D_x$ is obtained as the union of orbits through $D_x$. At this point of the proof we have located embedded solid tori inside $X$ and $Y$.

The connected sum is to be performed inside these tori. We use the existence of a diffeomorphism between $S^1\times D_x\#S^1\times D_y$ and $(S^1\times D_x \backslash S^{2k - 1}\times D^2),$ where the boundary component coming from the $S^{2k - 1}\times D^2$ piece that was carved out is identified with the boundary $\partial(S^1\times D_y) = S^1\times S^{2k - 1}$. To see the decomposition, split the $2k$-disk around $x$ as \begin{equation}D_x = D_{\epsilon_1}\cup (S^{2k - 1}\times [\epsilon_1, \epsilon_2]),\end{equation} where $D_{\epsilon_1}$ for small constants $\epsilon_1, \epsilon_2$. The 2-disk factor $\{pt\}\times D^2\subset S^{2k - 1}\times D^2$ is a small 2-disk around a point contained in the middle of $S^{2k - 1}\times [\epsilon_1, \epsilon_2]$ and transverse to $\{pt\}\times S^{2k - 1}\subset S^1\times S^{2k - 1}\times [\epsilon_1, \epsilon_2]$.

The manifold $X\# Y$ is constructed as the union of the pieces \begin{equation}{\label{Piece 1}}(X \backslash S^1\times D_x)\cup S^1\times D_{\epsilon}\end{equation} \begin{equation}{\label{Piece 2}}S^1\times S^{2k - 1}\times [\epsilon_1, \epsilon_2] \backslash (S^{2k - 1}\times D^2)\end{equation}and \begin{equation}{\label{Piece 3}}Y\backslash S^1\times D_y.\end{equation} The choices of $\mathcal{F}$-structures are the following. Equip pieces (\ref{Piece 1}) and (\ref{Piece 3}) with the initial $\mathcal{F}$-structure, which is polarized by hypothesis. Take a free circle action on the $S^{2k - 1}$ factor of $S^1\times S^{2k - 1}\times [\epsilon_1, \epsilon_2] \backslash (S^{2k - 1}\times D^2)$. This is the point of the proof where the hypothesis on the dimension of the manifolds to be $2k + 1 \geq 3$ is used. This equips piece (\ref{Piece 2}) with a polarized $\mathcal{T}$-structure. To see that these structures are compatible on the boundary components, we argue as follows. The action induced on each boundary component of pieces (\ref{Piece 1}), (\ref{Piece 2}), and (\ref{Piece 3}) by our choices of $\mathcal{F}$-structures pastes together to the canonical action on the circle factor into a $T^2$-action. Hence, the conditions of Definition \ref{Definition T-str} are satisfied, and the connected sum $X\# Y$ has a polarized $\mathcal{F}$-structure. 
\end{proof}

We sample in the following corollary the utility of the previous proposition to equip inequivalent smoothings with a polarized $\mathcal{F}$-structure. 

\begin{corollary}{\label{Corollary HRPN}}

\begin{itemize}

\item There are 28 inequivalent smoothings of $\R P^7$ that admit a polarized $\mathcal{T}$-structure. 

\item Any homeomorphism class that is realized by a closed flat 7-manifold has 28 inequivalent smoothings that admit a polarized $\mathcal{F}$-structure.\end{itemize}
\end{corollary}

To prove Corollary \ref{Corollary HRPN}, we use well-known argument to construct inequivalent smoothings on high-dimensional manifolds.

\begin{proof} There is a free circle action on the real projective 7-space, which yields a polarized $\mathcal{T}$-structure. Consider the connected sum $\R P^7\#\Sigma^7$, which is homeomorphic to the projective 7-space. Montgomery-Yang \cite{[MoYa]} have shown that every homotopy 7-sphere admits a pseudo free circle action. Example \ref{Example MY} and Proposition \ref{Proposition Connected} imply that any such connected sum admits a polarized $\mathcal{T}$-structure. We can choose homotopy 7-spheres $\Sigma^7_i$ and $\Sigma^7_j$ such that looking at the universal covers $S^7\#\Sigma_i\#\Sigma_i$ and $S^7\#\Sigma_j\#\Sigma_j$ of the two homeomorphic manifolds $\R P^7\# \Sigma^7_i$ and $\R P^7\# \Sigma^7_j$, we conclude that they are diffeomorphic if and only if $i = j$. The corollary now follows since there are 28 smoothings on $S^7$.

The proof of the second claim is similar and we present it without assumptions on the dimension. Let $\Sigma^n$ be a homotopy n-sphere and set $X = \mathbb{E}^n/\Gamma$. Cheeger-Gromov have shown that $X$ admits a polarized $\mathcal{F}$-structure \cite{[CG]}. The connected sum\begin{equation}\label{ExoticT}X\# \Sigma^n\end{equation} is homeomorphic to the Euclidean space form $\mathbb{E}^n/\Gamma$ but it need not be diffeomorphic to it (cf. \cite[Chapter 15]{[Wall]}). % Corollary \ref{Corollary Inertia}). It is proven in \cite{[FG0]} that the smoothings $X\#\Sigma$ are irreducible in the sense of \cite{[BaT]}.
\end{proof}

The next proposition generalizes \cite[Examples 0.3 and 1.3]{[CG]}. It will be used in Section \ref{Section Tori} and in the proof of Theorem \ref{Theorem P}.

\begin{proposition}{\label{Proposition OrbitSpace}} Let $\widetilde{M}$ and $M$ be closed smooth n-manifolds and let \begin{equation}\widetilde{M}\overset{\pi}\longrightarrow M\end{equation} be a finite covering. Suppose $\widetilde{M}$ admits a free torus action\begin{equation}\label{T-action}\widetilde{\phi}: T^k\rightarrow \Diff(\widetilde{M}).\end{equation}There is a polarized $\mathcal{F}$-structure on $M$.
\end{proposition}

\begin{proof} The free torus action $\widetilde{\phi}$ yields a pure polarized $\mathcal{T}$-structure on $\widetilde{M}$ with elements of the covering $\{V_1, \ldots, V_N\}$ given by embedded submanifolds that are diffeomorphic to the torus orbits. We use this structure to construct a polarized $\mathcal{F}$-structure on $M$. Regarding Item (1) of Definition \ref{Definition T-str}, choose a covering $\{U_1, \ldots, U_N\}$ for $M$ of open sets whose intersection are either empty or path-connected and such that elements of $\{\pi^{-1}(U_i)\}$ are diffeomorphic to $U_i$ for every $i\in \{1, \ldots, N\}$. Pick a Riemannian metric $g$ on $M$ and let $\tilde{g}$ be the metric on $\widetilde{M}$ that makes $\pi$ a Riemannian submersion $(\widetilde{M}, \tilde{g})\rightarrow (M, g)$. Fix points $x\in M$ and $y\in \pi^{-1}(x)\subset \widetilde{M}$, and take an $x_i\in U_i$ for every $i$. Let $\alpha_i:[0, 1]\rightarrow M$ be a geodesic with initial and end points $\alpha_i(0) = x$ and $\alpha_i(1) = x_i$. There is a geodesic $\widetilde{\alpha}_i: [0, 1]\rightarrow \widetilde{M}$ with $\widetilde{\alpha}_i(0) = y$ and $\widetilde{\alpha}_i(1) = \widetilde{x}_i\in \pi^{-1}(x_i)$ and $\widetilde{x}_i\in \widetilde{U}_m$ for some open subset $\widetilde{U}_i\subset \widetilde{M}$. We now use the $T^k$-torus action of (\ref{T-action}) to define an action $\phi_i: T^{k_i}\rightarrow \Diff(\widetilde{U}_i)$ with trivial isotropy group as in Items (3) and (7) of Definition \ref{Definition T-str}. Notice that the hypothesis on $\widetilde{\phi}$ can be relaxed as follows. For any action $T^k\rightarrow \Diff(\widetilde{M})$ whose isotropy group has codimension greater than one, we immediately obtain a free circle action for a given $S^1\subset T^k$.  We can define a circle action as the aforementioned restriction of the torus action $\widetilde{\phi}: T^k\rightarrow \Diff(V_j)$ to $\widetilde{U}_i\subset \widetilde{M}$ for every $1\leq i \leq N$. If the intersections $U_i\cap U_j$ and $\widetilde{U}_i\cap \widetilde{U}_j$ are both nonempty, we have a canonical covering. If the intersection $U_i\cap U_j$ is empty, yet $\widetilde{U}_i\cap \widetilde{U}_j\neq \emptyset$, then there is a Deck transformation $g: \widetilde{M}\rightarrow \widetilde{M}$ such that $g\widetilde{U}_i\cap \widetilde{U}_j\neq \emptyset$. The corresponding group of Deck transformations on the open set is finite by assumption and $\Aut(S^1) \cong \Z/2$.. The commutativity condition in Item (4), and invariance of Item (5) of Definition \ref{Definition T-str} are immediately satisfied. Hence, we have constructed a polarized $\mathcal{F}$-structure on $M$.
\end{proof}

\subsection{Collapsing, minimal volumes and entropy, and Yamabe invariant}{\label{Section CG}}A smooth manifold $M$ collapses with bounded sectional/Ricci/scalar curvature if and only if there exists a sequence of Riemannian metrics $\{g_j\}$ for which the sectional/Ricci/scalar curvature is uniformly bounded, but their volumes $\{\Vol(M, g_j)\}$ approach zero as $j\rightarrow \infty$. If the curvature of $\{g_j\}$ only has a lower bound, we say that $M$ collapses with curvature bounded from below.\\

Define\begin{equation}\MinVol(M):= \underset{g}{\inf}\{\Vol(M, g) : |K_g|\leq 1\}\end{equation}and\begin{equation}\Vol_K(M):= \underset{g}{\inf}\{\Vol(M, g): K_g \geq - 1\},\end{equation} where the sectional curvature of the Riemannian metric $g$ is denoted by $K_g$, and we assumed the normalization $\Vol(M, g) = 1$. The invariant $\MinVol(M)$ is known as the minimal volume and was introduced by Gromov \cite{[G]}.\\

For a real number $D > 0$, the more restrictive invariant $\DMinVol$(M) of a closed smooth Riemannian manifold $(M, g)$ can be defined as follows. Much like $\MinVol(M)$, it is required that the infimum of $\Vol(M, g)$ is to be taken over all metrics $g$ and additionally requiring that the diameter remains bounded from above by $D$. The minimal entropy $\He(M)$ is the infimum of the topological entropy of the geodesic flow of a smooth metric $g$ on $M$ such that $\Vol(M, g) = 1$. The inequality \begin{equation}[\He(M)]^n \leq (n - 1)^n \MinVol(M)\end{equation} is known to hold \cite[page 417]{[PP1]}.\\

Let us now recall the definition of the Yamabe invariant  \cite{[BeY], [ScY]}. Let\begin{equation}\gamma:= [g] = \{ug : M\overset{u}{\rightarrow} \R^+\}\end{equation} be a conformal class of Riemannian metrics on $(M, g)$. The Yamabe constant of $(M, \gamma)$ is\begin{equation} \mathcal{Y}(M, \gamma):= \underset{g\in \gamma}{\inf} \frac{\int_M Scal_g d\Vol_g}{(\Vol(M, g))^{\frac{n - 2}{n}}}.\end{equation} The Yamabe invariant of $M$ is defined as\begin{equation} \mathcal{Y}(M):= \underset{\gamma}{\sup}  \mathcal{Y}(M, \gamma).\end{equation} 

A profound utility of $\mathcal{F}$-structures is sampled in the following results.

\begin{theorem}\label{Theorem CG} Cheeger-Gromov \cite{[CG]}. The minimal volume of a manifold $M$ that admits a polarized $\mathcal{F}$-structure vanishes, and $M$ collapses with bounded sectional curvature.
\end{theorem}

\begin{theorem}\label{Theorem PP1} Paternain-Petean \cite{[PP1]}. If a manifold $M$ admits an $\mathcal{F}$-structure then $\He(M) = 0 = \Vol_K(M)$ and $M$ collapses with sectional curvature bounded from below.
\end{theorem}

Collapse with scalar curvature bounded from below and collapse with bounded scalar curvature are equivalent for manifolds of dimension at least three \cite[Proposition 7.1]{[PP1]}. The latter is equivalent to the Yamabe invariant being nonnegative. The Yamabe invariant of a smooth closed manifold $M$ is positive if and only if $M$ admits a Riemannian metric of positive scalar curvature \cite{[BeY]}. In particular, we have the following well-known lemma \cite{[BeY], [Lebrun], [PP1]}.

\begin{lemma}\label{Lemma Blackbox} Suppose $M$ is a closed smooth manifold that does not admit a Riemannian metric of positive scalar curvature. If $M$ collapses with bounded scalar curvature, then\begin{equation}\mathcal{Y}(M) = 0.\end{equation}Moreover, any scalar-flat Riemannian metric on $M$ is Ricci-flat.
\end{lemma}

\subsection{Non-realization of Yamabe invariant zero in our examples}{\label{Section NY}}  A compact orientable Riemannian n-manifold $M$ is $\epsilon^{-1}$-hyperspherical if there is a map onto the unit n-sphere that multiplies all distances by a factor less than or equal to $\epsilon > 0$. If for every $\epsilon > 0$, there exists a finite covering of $X$ that is $\epsilon^{-1}$-hyperspherical and spin, then $X$ is enlargeable \cite[Section 0]{[GL]}. Compact Euclidean space forms are examples of enlargeable manifolds, and enlargeability is a homotopy invariant. The non-realization of the Yamabe invariant for our examples in Theorem \ref{Theorem P} and Corollary \ref{Corollary P} is a corollary of the following result (cf. Lemma \ref{Lemma Blackbox}).

\begin{theorem}{\label{Theorem BGL}} Bieberbach,  Gromov-Lawson \cite{[GL]} (cf. Schoen-Yau \cite{[SCYA]}). There are no Riemannian metrics of positive scalar curvature on a homotopy Euclidean space form $M$. Any Riemannian metric of non-negative scalar curvature on a homotopy Euclidean space form is flat. In such a case, $M$ is diffeomorphic to $\mathbb{E}^n/\Gamma$.
\end{theorem}

We finish this brief section with a pair of sample results.

\begin{theorem}{\label{Theorem HP}} Hitchin \cite{[NH]}, Petean \cite{[JP]}. Let $\Sigma$ be a homotopy n-sphere such that $\alpha(\Sigma)\neq 0$. Its Yamabe invariant is \begin{equation} \mathcal{Y}(\Sigma) = 0\end{equation}and it is not realized, i.e., there are no scalar-flat Riemannian metrics on $\Sigma$.
\end{theorem}

\begin{proof} Petean has shown that the Yamabe invariant of a closed simply connected n-manifold with $n\geq 5$ is nonnegative \cite{[JP]}. Hitchin has shown that $\alpha(\Sigma)\neq 0$ implies that $\Sigma$ does not admit a Riemannian metric of positive scalar curvature \cite{[NH]}. Since the Yamabe invariant of a closed manifold is positive if and only if the manifold admits a Riemannian metric of positive scalar curvature, these results imply the vanishing $\mathcal{Y}(\Sigma) = 0$. To argue that the Yamabe invariant is unrealized, we proceed by contradiction. Suppose there exists a scalar-flat Riemannian metric $g$ on $\Sigma$. Since $\alpha(\Sigma)\neq 0$ by assumption, the homotopy sphere has a non-trivial parallel spinor, and the Riemannian manifold $(\Sigma, g)$ thus has special holonomy \cite{[NH]}. This is a contradiction, since homotopy spheres have generic holonomy. Hence, $\Sigma$ does not admit a scalar-flat Riemannian metric, and its Yamabe invariant is not realized.

\end{proof}

It is unknown if the minimal entropy  depends on the choice of smooth structure. In this paper, we provide a myriad of examples in terms of homeomorphism classes where it does not.

\begin{proposition}{\label{Proposition Proj}} There exists a smoothing $M$ of $\mathbb{CP}^5$ such that \begin{equation}\He(M) = 0 = \Vol_K(M),\end{equation} and $M$ collapses with sectional curvature bounded from below.

Its Yamabe invariant is \begin{equation} \mathcal{Y}(M) = 0,\end{equation} and it is not realized, i.e., there are no scalar-flat Riemannian metrics on $M$.

\end{proposition}

\begin{proof} Let $\Sigma$ be a homotopy 10-sphere such that $\alpha(\Sigma)\neq 0$; this sphere was proven to exist in \cite{[Mil3], [Ad]}. The smooth manifold $M:= \mathbb{CP}^5\# \Sigma$ is homeomorphic to the complex projective 5-space. There is a nontrivial circle action on $\mathbb{CP}^5$ and on $\Sigma$ (see \cite{[RS2]}), and \cite[Theorem 5.9]{[PP1]} implies that there is a $\mathcal{T}$-structure on $M$. Theorems \ref{Theorem PP1} implies that the minimal entropy of $M$ is zero, and that it collapses with sectional curvature bounded from below. In particular, $\mathcal{Y}(M) \geq 0$. This also follows from Petean's result \cite{[JP]}. Since $\alpha(M) = \alpha(\mathbb{CP}^5\# \Sigma) = \alpha(\mathbb{CP}^5) + \alpha(\Sigma)$ \cite{[NH]}, it follows that $\mathcal{Y}(M) = 0$. The claim about the nonexistence of scalar-flat Riemannian metrics on $M$ follows from the argument given in the proof of Theorem \ref{Theorem HP} (cf. Lemma \ref{Lemma Blackbox}).
\end{proof}

\subsection{Homotopy spheres and homotopy real projective spaces}\label{Section Spheres}Existence questions of circle actions on homotopy spheres has been a classical area of research since the late 1960's see \cite{[MoYa], [RS2]} and references therein. These results yield We now mention some more constructions of $\mathcal{T}$-structures on closed smooth manifolds on the homeomorphism class of spheres and on the homotopy equivalence class of real projective spaces. A $\mathcal{T}$-structure can be constructed on any homotopy odd-dimensional sphere that bounds a parallelizable manifold as follows. The group of homotopy spheres that bound parallelizable manifolds are generated by the Kervaire spheres $\Sigma^{2n - 1}(d)$ \cite{[KeM]}. All Kervaire spheres arise as Brieskorn manifolds in the following manner \cite{[Br], [Hir]}. Suppose $n\geq 3$ is an odd integer number, and $d = \pm 3$ mod 8. Then a homotopy sphere $\Sigma^{2n - 1}(d)$ arises as a real algebraic submanifold of $\C^{n + 1}$ defined by the equations\begin{equation}z_0^d + z_1^2 + \ldots + z_n^2 = 0 \end{equation}and\begin{equation} |z_0|^2 + |z_1|^2 + \ldots + |z_n|^2 = 1.\end{equation}Scalar multiplication on $\C^{n + 1}$ yields a circle action on $\Sigma^{2n - 1}(d)$ and using a result of Paternain-Petean \cite[Theorem 5.1]{[PP1]}. we conclude that every homotopy odd-dimensional sphere that bounds a parallelizable manifold admits a $\mathcal{T}$-structure. A polarized $\mathcal{T}$-structure on these manifolds is constructed using pseudo-free circle actions as in Example \ref{Example MY}. Montgomery-Yang have shown that all homotopy 7-spheres admit such an action \cite{[MoYa]}.\\

The $\mathcal{T}$-structure on the Brieskorn manifolds descend to a $\mathcal{T}$-structure on closed manifolds that are homotopy equivalent to $\R P^{2n - 1}$. Indeed, there is an orientation-preserving fixed point involution $T: \Sigma^{2n - 1}(d)\rightarrow \Sigma^{2n - 1}(d)$ given by \begin{equation}(z_0, z_1, z_2,\ldots, z_{n - 1}, z_n)\mapsto (z_0, - z_1, - z_2, \ldots, -z_{n - 1}, -z_n).\end{equation}The circle action commutes with the involution and yields a $\mathcal{T}$-structure on the orbit space $\Sigma^{2n - 1}(d)/T$. The latter is a closed manifold that is homotopy equivalent to $\R P^{2n - 1}$ yet it need not be homeomorphic nor diffeomorphic to it \cite{[LMe]}.\\

Another construction of $\mathcal{T}$-structures arises as a generalization of Example \ref{Example SR} in terms of a broader set of choices for the decomposition of $S^n$ into two building blocks that was given in (\ref{Decomposition S1}), and the different choices of diffeomorphisms that are used to identify the pieces together. Consider now\begin{equation}\label{Decomposition NewS1}S^n = \partial D^{n + 1} = \partial (D^k\times D^{n + 1 - k}) = (S^{k - 1}\times D^{n + 1 - k})\cup_{\id}(D^k\times S^{n - k}).\end{equation}Build a homotopy sphere\begin{equation}\label{HomFst}\Sigma^n = (S^{k - 1}\times D^{n + 1 - k})\cup_{\varphi_{k, n}} (D^k\times S^{n - k})\end{equation}such that the diffeomorphism of the common boundaries\begin{equation}\varphi_{k, n}: S^{k - 1}\times S^{n - k}\rightarrow S^{k - 1}\times S^{n - k}\end{equation} satisfies $\pi_{S^{k - 1}}\circ \varphi_{k, n} = \pi_{S^{k - 1}}$, where $\pi_{S^{k - 1}}$ is the projection to the sphere factor \cite{[GW]}. A $\mathcal{T}$-structure is given by taking the open cover $U_1:= S^{k - 1}\times D^{n + 1 - k}$ and $U_2: = D^k\times S^{n - k}$ and choosing torus actions that commute with the gluing diffeomorphism. The Gromoll groups indicate when the homotopy n-sphere (\ref{HomFst}) is exotic \cite[Section 3]{[GW]}.

\subsection{Homotopy Euclidean space forms}\label{Section Tori}  The main result of this section is Theorem \ref{Theorem Davis}. In order to share light on it, we begin by discussing homotopy tori. As it was mentioned in the introduction, Casson, Wall and Hsiang-Shaneson have shown existence of inequivalent smoothings (up to diffeomorphism) and inequivalent PL-structures (up to PL-homeomorphism) on homotopi tori. The fundamental theorem of smoothing theory states that a topological manifold of dimension at least five admits a smooth structure if and only if its topological tangent bundle admits the structure of a vector bundle. A homotopy n-torus with a PL-structure is parallelizable, hence smoothable \cite[Chapter 15A]{[Wall]}. The set $\mathcal{S}^{PL}(T^n)$ of equivalence classes up to PL-homeomorphism of closed n-dimensional PL-manifolds that are homotopy equivalent to the n-torus is in bijective correspondence with $H^3(T^n; \Z/2)$ \cite{[W], [HS], [Wall]}. There is an isomorphism \begin{equation} [T^n, G/PL]\cong \underset{0\leq i\leq n}\bigoplus H^i(T^n; \pi_i(G/PL)),\end{equation} where the groups $\pi_i(G/PL)$ are finite. According to Wall \cite[Corollary]{[W]} \cite[Page 236]{[W]} and Hsiang-Shaneson \cite[Theorem B]{[HS]}, a closed n-dimensional PL-manifold that is homotopy equivalent to the n-torus with $n\geq 5$ is finitely covered by the standard $T^n$. The normal PL-invariants are natural for coverings and any element of the group $H^r(T^n; A)$ with $A$ finite is killed by passing to a finite cover (an argument to show this will be given in the proof of Theorem \ref{Theorem Davis}). Coverings corresponding to subgroups of $\pi_1 = \Z^n$ of odd index are fake tori; those that correspond to subgroups of even index are standard. The set of smoothings on a PL-structure class are in bijective correspondence with elements of \begin{equation}[T^n, PL/O]\cong \underset{i\leq n}\bigoplus H^i(T^n; \pi_i(PL/O))\end{equation} and the normal invariants, as in the PL-structures case, are natural for coverings \cite{[Wall]}; indeed, the homotopy groups of the loop space $PL/O$ are finite and any element of the cohomology group of a homotopy n-torus with coefficients in the finite group $\pi_i(PL/O)$ can be killed by passing to a finite cover. We are indebted to Terry Wall for explaining this to us. Jim Davis \cite[Appendix]{[Davis]} has an excellent synthesized description of these properties, and his synthesis will be used in the proof of Theorem \ref{Theorem Davis}. Corollary \ref{Corollary P} and Proposition \ref{Proposition P} are proven using the following result.

\begin{theorem}{\label{Theorem FCTori}} Hsiang-Shaneson \cite{[HS]}, Wall \cite{[W], [Wall]}. For $n\geq 5$, every homotopy n-torus is finitely covered by the standard $T^n$.

\end{theorem}

Let us now move into homotopy Euclidean space forms in full generality. Regarding the structure of their topological tangent bundle, the reader is directed to \cite{[JTh]} and we now consider smoothable homotopy Euclidean space forms. We now present the main result in this section, which is a consequence of the aforementioned results of Hsiang-Shaneson and Wall.

\begin{theorem}\label{Theorem Davis}Let $M$ be a homotopy Euclidean space form $\mathbb{E}^n/\Gamma$. There is a homeomorphism\begin{equation}\label{Homeo FJ}\varphi:M\rightarrow \mathbb{E}^n/\Gamma\end{equation}and a finite cover\begin{equation}\label{Cover Bieberbach}T^n\rightarrow \mathbb{E}^n/\Gamma\end{equation}so that the induced pullback homeomorphism\begin{equation}\label{Basic Point}\widetilde{M}\rightarrow T^n\end{equation}is isotopic to a diffeomorphism.
\end{theorem}

The existence of the homeomorphism (\ref{Homeo FJ}) was established by Farrell-Jones \cite{[FJo]}, where they showed that any homotopy equivalence between such manifolds is isotopic to a homeomorphism. One of the Bieberbach Theorems states the existence of the finite covering (\ref{Cover Bieberbach}). It is immediate to see that Theorem \ref{Theorem Davis} follows from Theorem \ref{Theorem FCTori}. The justification for including the following proof is to flesh out details of the results in surgery theory that are being employed. We built greatly on work of Jim Davis \cite[Appendix]{[Davis]} for the following description.  

\begin{proof} We use work of Jim Davis to prove that the pullback homeomorphism (\ref{Basic Point}) is isotopic to a diffeomorphism. Two homeomorphisms are isotopic if there exists a homotopy between them that consists of homeomorphisms. We commence by recalling a necessary property.

\begin{condition}\label{Definition Condition1} For any $i > 0$, for any finite abelian group $A$, for any finite cover $\hat{p}: \hat{N}\rightarrow N$, and for any element $x\in H^i(\hat{N}; A)$, there exists a finite cover $\tilde{p}: \widetilde{N}\rightarrow \hat{N}$ such that $\tilde{p}^{\ast}x = 0$.
\end{condition}

An extension of the described at the beginning of the section for homotopy tori is the following lemma, which is the main ingredient to prove Theorem \ref{Theorem Davis}.

\begin{lemma}\label{Lemma Davis} Davis \cite[Lemma A.3]{[Davis]}. Let $\varphi: M\rightarrow N$ be a homeomorphism of closed smooth n-manifolds with $n\geq 5$, and assume $N$ satisfies Condition \ref{Definition Condition1}. There exists a finite cover \begin{equation}\hat{N}\rightarrow N\end{equation}such that the induced pull-back homeomorphism \begin{equation}\label{CoveringHNNew}\hat{M}\rightarrow \hat{N}\end{equation} is isotopic to a diffeomorphism.
\end{lemma}

To prove Theorem \ref{Theorem Davis}, we need to show that a homotopy Euclidean form satisfies Condition \ref{Definition Condition1} in order to invoke Lemma \ref{Lemma Davis}.

\begin{claim}\label{Claim Davis} Any Euclidean space form $\mathbb{E}^n/\Gamma$ satisfies Condition \ref{Definition Condition1}.
\end{claim}

One of the Bieberbach theorems states that there is a finite covering of $\mathbb{E}^n/\Gamma$ that is isometric to the n-torus. Thus, in order to prove Claim \ref{Claim Davis}, it suffices to check that $T^n$ satisfies the following condition.

\begin{condition}\label{Definition Condition2} For any $i > 0$, for any finite abelian group $A$, and for any element $x\in H^i(N; A)$, there exists a finite cover $\hat{p}: \hat{N}\rightarrow N$ such that $\hat{p}^{\ast}x = 0$.
\end{condition}

We proceed to show that the n-torus satisfies Condition \ref{Definition Condition2} with $\hat{N} = T^n = N$. The following argument fills in the details of the description of the normal invariants of a homotopy n-torus given at the beginning of the section. For $i = 1$, we have the commutative diagram\begin{equation}\begin{array}[c]{cccc}
H^1(T; A)&\stackrel{\cong}{\longrightarrow}& Hom(\pi_1(T^n), A) =& Hom(\Z^n, A)\\
\downarrow\scriptstyle{\hat{p}^{\ast}}&&\downarrow\scriptstyle{- \circ \hat{p}_{\ast}}\\
H^1(T^n; A)&\stackrel{\cong}{\longrightarrow}& Hom(\pi_1(T^n), A) =& Hom(\Z^n, A)
\end{array}\end{equation}
that show $\hat{p}^{\ast}x = 0$ for every element $x\in H^1(T^n, A)$.

We now argue in the case $i\geq 2$ by induction on $n$. A free circle action on $T^n$ gives rise to a principal $S^1$-bundle $S^1\rightarrow T^n\rightarrow T^n/S^1 = T^{n - 1}$. Suppose that Condition \ref{Definition Condition2} holds for any torus of smaller dimension. Canonically associated to the principal $S^1$-fibration we have \cite{[DavKirk]}\begin{equation}\cdots\longrightarrow H^{i - 2}(T^{n - 1}; A)\overset{\cup e}{\longrightarrow} H^1(T^{n - 2}; A)\overset{\pi^{\ast}}{\longrightarrow} H^i(T^n; A)\overset{\pi_!}\longrightarrow H^{i - 1}(T^{n - 1}; A)\longrightarrow \cdots \end{equation}the exact Gysin sequence. Our hypothesis on the dimension implies the existence of a cover $\tilde{p}: T^{n - 1}\rightarrow T^{n - 1}$ so that $\tilde{p}^{\ast}(\pi_! x) = 0$ for any element $H^i(T^{n - 1}; A)$ with $i\geq 2$. 
Consider now the diagram\begin{equation}\begin{array}[c]{cccc}
T^n&\stackrel{\cong}{\longrightarrow}& T^n/S^1 =& T^{n - 1}\\
\downarrow\scriptstyle{\tilde{p}}&&\downarrow\scriptstyle{\tilde{p}_{S^1}}\\
T^n&\stackrel{\cong}{\longrightarrow}& T^n/S^1 =& T^{n - 1}.
\end{array}\end{equation}It yields a map of principal $S^1$-bundles and a map of exact Gysin sequences. It is now immediate to conclude that $T^n$ satisfies Condition \ref{Definition Condition2}, thus $\mathbb{E}^n/\Gamma$ satisfies Condition \ref{Definition Condition1}, and Lemma \ref{Lemma Davis} is ready to be invoked. This concludes the proof of Theorem \ref{Theorem Davis}.
\end{proof}

%We mention a consequence of the previous result to the connected sum construction of inequivalent smoothings using homotopy spheres as in Corollary \ref{Corollary HRPN}. The following result generalizes \cite[Proposition ]{[FG0]}. The group of homotopy n-spheres is denoted by $\Theta_n$ \cite{[KeM]}.

%\begin{corollary}\label{Corollary Inertia} Let $M$ be a homotopy Euclidean space form of dimension at least five. Suppose the finite covering (\ref{CoveringHNNew}) is $k$-sheeted, and let $\Sigma\in \Theta_n$ be a homotopy n-sphere of order $d$. The connected sum $M\# \Sigma$ is not diffeomorphic to $M$. 
%\end{corollary}

%\begin{proof}We proceed by contradiction and assume there is a $[\Sigma] \neq 0 \in \Ine(M)$ with $\Sigma$ not diffeomorphic to the n-sphere. Hence, a finite cover of $X\# \Sigma^n$ is\begin{equation}\label{ExampleExoticCover}T^n\#\underbrace{\Sigma^n\#\cdots \#\Sigma^n}_k,\end{equation}which is diffeomorphic to $T^n$ if and only if $\Sigma^n$ is diffeomorphic to $S^n$ 

%\end{proof}

We now proceed to prove that any smooth manifold homeomorphic to a high-dimensional torus can be equipped with a polarized $\mathcal{F}$-structure. 

\begin{corollary}{\label{Corollary EucSF}} Every homotopy Euclidean n-space form with $n\geq 5$ admits a polarized $\mathcal{F}$-structure.

In particular, every smoothing and every PL-structure on the homeomorphism type of $T^n$ with $n\geq 5$ admits a polarized $\mathcal{F}$-structure. 
\end{corollary}

\begin{proof} The proposition follows from Proposition \ref{Proposition OrbitSpace} and Theorem \ref{Theorem Davis}. 
\end{proof}

Examples of inequivalent smoothings on homotopy Euclidean space forms were described in Corollary \ref{Corollary HRPN}. We finish the section by mentioning an explicit construction of inequivalent PL-structures on homotopy Euclidean space forms.

\begin{example}\label{Example FakeTori} Fake tori. \emph{The connected sum construction (\ref{ExoticT}) with $X = T^n$ of homotopy n-tori that was described in Corollary \ref{Corollary HRPN} produces inequivalent smoothings but not inequivalent PL-structures. For any homotopy sphere $\Sigma^n$, the manifolds $T^n\#\Sigma^n$ and $T^n$ are PL-homeomorphic by the Alexander trick. Fake n-tori are constructed as mapping tori\begin{equation}M_h:= [0, 1]\times T^n/(1, x) (0, h(x))\end{equation} for a diffeomorphism $h: T^n\rightarrow T^n$, as described in Farrell-Gogolev \cite{[FG]}. If the diffeomorphism $h$ is not PL pseudo-isotopic to the identity map $id_{T^n}$, then the mapping torus $M_h$ is a fake torus \cite[Corollary 2.3]{[FG]}. The fake torus $M_h$ is the total space of a $T^n$-bundle over the circle. There is a bijection between the smooth pseudo-isotopy classes of diffeomorphisms $h$ whose induced homomorphisms on fundamental groups satisfy $h_{\ast} = id_{\pi_1(T^n)}$, and the smooth structures $\theta$ on a homotopy torus $M$ such that the inclusion map $\sigma: T^n\times \{0\}\hookrightarrow (M\times S^1, \theta)$ is a smooth embedding. The smooth structures correspond to the subgroup of elements $\varphi\in [T^{n + 1}, TOP/O]$ such that $\varphi\circ\sigma$ is null-homotopic \cite{[Wall], [FG]}. It is straight-forward to construct a polarized $\mathcal{T}$-structure on $M_h$ using the constructions presented in Section \ref{Section TC}.}
\end{example}

\section{Proofs of results stated in the introduction}

\subsection{Proof of Theorem \ref{Theorem P} and Corollary \ref{Corollary P}}{\label{Section Proof P}}

Homotopy Euclidean space n-forms admit a polarized $\mathcal{F}$-structure by Proposition \ref{Proposition OrbitSpace} and Theorem \ref{Theorem Davis} as stated in Corollary \ref{Corollary EucSF}. Their minimal volume is zero and they collapse with bounded sectional curvature by Theorem \ref{Theorem CG}, which implies that their Yamabe invariant is nonnegative as indicated in Lemma \ref{Lemma Blackbox}. Such a manifold does not admit a Riemannian metric of positive scalar curvature \cite{[GL]}, as stated in Theorem \ref{Theorem BGL}. Hence, we conclude that a homotopy Euclidean space form has zero Yamabe invariant. Other than the standard $\mathbb{E}^n/\Gamma$, no homotopy Euclidean space form admits a scalar-flat Riemannian metric by Theorem \ref{Theorem BGL}. 

\begin{remark}\label{Remark NilDavis} Jim Davis has shown in \cite[Appendix]{[Davis]} that a smooth manifold of dimension at least five and that is homeomorphic to a nilmanifold has a finite covering which is diffeomorphic to a nilmanifold. The existence of a polarized $\mathcal{F}$-structure on such a manifold can be proven using tweaks to the arguments and constructions that were presented in Section \ref{Section TC}. The Yamabe invariant on this homeomorphism class is zero and it is not realized by any smoothing since a smooth manifold that is homeomorphic to a nilmanifold does not admit Ricci-flat metrics.

\end{remark}

\subsection{Proof of Proposition \ref{Proposition P}}{\label{Section ProofPropP}}  This result is a produce of the work of Baues-Tuschmann \cite{[BaT]}, Theorem \ref{Theorem P}, and Theorem \ref{Theorem FCTori}. Indeed, let $M$ be an exotic/fake torus that is not a product of standard and a lower-dimensional exotic/fake tori; exotic tori of this kind are considered in Corollary \ref{Corollary HRPN} and Example \ref{Example FakeTori}. Baues-Tuschmann have shown that \begin{equation}\DMinVol(M) > 0.\end{equation} On the other hand, \begin{equation}\DMinVol(T^n) = 0.\end{equation} Since any homotopy n-tori is finitely covered by the standard n-torus, the proposition follows.

\begin{remark}{\label{Remark TS}}Baues-Tuschmann asked in \cite[Question 1.7]{[BaT]} if an exotic/fake torus that is not a product of a standard torus and a lower-dimensional exotic/fake torus always has non-vanishing minimal volume. Corollary \ref{Corollary P} answers their question in the negative. As we have mentioned in the introduction, they have shown that the finer invariant $\DMinVol$ does discern the standard smooth structure within a subset of homotopy tori. Corollary \ref{Corollary P} states that it is the realization of the Yamabe invariant that tells the standard smooth structure apart from any other.  
\end{remark}

\subsection{Proof of Theorem \ref{Theorem FundGroup}}  A polarized $\mathcal{T}$-structure was constructed on $M(G)$ in Proposition \ref{Proposition PGPolarized}. The claims concerning its minimal volume and collapse follow from Theorem \ref{Theorem CG}.

\section{Acknowledgements}

We are indebted to Jim Davis for kindly bringing \cite[Appendix]{[Davis]} to our attention, without which we could have not proven our main result as stated. We are grateful to an anonymous referee for kindly pointing out a mistake on a previous version of the manuscript. We thank Frank Quinn, Andrew Ranicki, and Terry Wall for helpful e-mail correspondence. We thank Bernd Ammann, Oliver Baues, Fernando Galaz-Garc\'ia, Marco Radeschi, and Wilderlich Tuschmann for interesting conversations. We gratefully acknowledge support from the University of Fribourg and the organizers of its Riemannian Topology Seminar 2015 for a very pleasant and productive meeting during which part of this paper was written.

\end{document}